%% file: conv_order_one.tex
\documentclass[reqno]{amsart}
\makeatletter
\@namedef{subjclassname@2010}{%
  \textup{2010} Mathematics Subject Classification}
\makeatother

\input{preambles}

\usepackage{tikz-cd}
\usepackage{cite}
\usepackage{standalone}
\newcommand\numberthis{\addtocounter{equation}{1}\tag{\theequation}}

\newcommand\ovr[1]{\overrightarrow{#1}}

\newcommand\wN{\widetilde{N}}

\newcommand\wL{\widetilde{L}}

\begin{document}

\title[convolution of gamma and negative binomial variables]{Stochastic orders for convolution of heterogeneous gamma and negative binomial random variables}

\author{Ying Zhang}

\address[Y. Zhang]{Susquehanna International Group, Bala Cynwyd, PA 19004}

\email{zhangyingmath@gmail.com}

\date{\today}



%\emph{AMS 2010 subject classification}: primary 60E15, secondary 60K10

% \subjclass[2010]{primary 60E15, secondary 60K10}

% \keywords{likelihood ratio order, reversed hazard rate order, convolution order, total positivity, majorization, gamma convolution, negative binomial convolution}

\begin{abstract}
Convolutions of independent random variables often arise in a natural way in many applied problems. In this article, we compare convolutions of two sets of gamma (negative binomial) random variables in the convolution order and the usual stochastic order in a unified set-up, when the shape and scale parameters satisfy a partial order called reverse-coupled majorization order. This partial order is an extension of the majorization order from vectors to pairs of vectors which also incorporates the arrangement information. The results established in this article strengthen those known in the literature. 
\end{abstract}

\maketitle

\section{Introduction}
%Since the distribution theory is quite complicated when the convolution involves independent and non-identical random variables, it is of great interest to investigate stochastic properties of convolutions and derive bounds and approximations on some characteristics of interest in this setup.  There are a large nummber of extensive studies on this topic in the literature. Moreover, the gamma distribution can be widely applied in actuarial science as most total insurance claim distributions have quite similar shape to gamma distributions: non-negatively supported, skewed to the right and unimodal see \cite{bowers1997actuarial}\cite{FURMAN20082353}. 

Convolutions of independent random variables often arise in a natural way in many applied areas including applied probability, reliability theory, actuarial science, nonparametric goodness-of-fit testing, and operations research. Since the distribution theory is quite complicated when the convolution involves INID (independent and non-identical) random variables, it is of great interest to derive bounds and approximations in the form of some stochastic orderings. Typical applications of such bounds could be found in \cite{Diaconis1990}\cite{Boland1994}\cite{Ma2000convex}\cite{Roosta2015schur} and references therein. One of the most commonly used stochastic ordering is the \textit{usual stochastic order}: random variables $X_1$ is said to be smaller than $X_2$ in this order, denoted by $X_1\leq_{st} X_2$ if $X_2$ has the same distribution as $X_1 + Z$, where $Z$ is nonnegative. The fact that the usual stochastic order is useful can readily be seen from its equivalent characterizations \cite{Shaked2007stochastic}:
\begin{align*}
& X_1 \leq_{st} X_2 \iff \forall t, \quad P(X_1\geq t) \leq P(X_2\geq t) \iff \\
& \forall \phi \mbox{ increasing and where the integrals are finite},\quad E[\phi(X_1)] \leq E[\phi(X_2)] 
\end{align*}

A closely related order is the \textit{convolution order}: $X_1\leq_{conv} X_2$ if $X_2\stackrel{st}{=} X_1 + Z$, where $Z$ is nonnegative and independent of $X_1$. As is evident from the definition, the convolution order is a stronger order that implies the usual stochastic order. The convolution order is useful for the purpose of comparison of experiments, see \cite{Shaked2003convolution}. As we shall see in this article, the convolution order is a natural order to study for convolutions of infinitely divisible random variables, such as the gamma random variable and its discrete analogue, the negative binomial random variable. 

% where it can be shown that the more reliable a random variable is in terms of the convolution order, the less information it gives about an underlying location parameter, see \cite[Proposition 1, 2]{Shaked2003convolution} for a full discussion. It is also pointed out there that the convolution order is a realistic assumption for some nonparametric inferential problems. 

It is well known that gamma distribution is one of the most commonly used distributions in statistics, reliability and life testing. Stochastic orderings for convoluions of heterogeneous gamma random variables are extensively studied in the literature, see \cite{Proschan1976}\cite{nevius1977}\cite{Bock1987}\cite{Boland1994}\cite{Bon1999ordering}\cite{Kochar1999dispersive}\cite{Korwar2002} \cite{Khaledi2004}\cite{Yu2009}\cite{Zhao2009mean}\cite{Mao2010}\cite{Zhao2010pascal}\cite{Kochar2011tail}\cite{Li2013stochastic}\cite{Zhao2014ordering} and the references therein. Let $G_{\alpha, \beta}$ be a gamma random variable with shape parameter $\alpha$ and scale parameter $\beta$. Then, in its standard form $G_{\alpha, \beta}$ has the probability density function
\[
f_{G}(t)= \begin{cases}
\frac{\beta^{\alpha}t^{\alpha-1}}{\Gamma(\alpha)}e^{-\beta t}, & t\geq 0\\
0, & t < 0
\end{cases}
\]

$G_{\alpha, \beta}$ is a flexible family of distributions with decreasing, constant, and increasing failure rates when $0 < \alpha < 1$ , $\alpha = 1$ and $\alpha > 1$, respectively. Given $\ovr{\alpha_1}=(\alpha_{11}, \alpha_{12}, \ldots, \alpha_{1n})$, $\ovr{\alpha_2}=(\alpha_{21}, \alpha_{22}, \ldots, \alpha_{2n})$, $\ovr{\beta_1}=(\beta_{11}, \beta_{12}, \ldots, \beta_{1n})$, $\ovr{\beta_2}=(\beta_{21}, \beta_{22}, \ldots, \beta_{2n})$, we adopt the following notation to denote convolutions of gamma random variables 
\[
\cG_r := \cG_{\ovr{\alpha_r}, \ovr{\beta_r}}:=\sum_{i=1}^n G_{\alpha_{ri}, \beta_{ri}}, \quad r=1,2
\]
% \[
% \cG_2 := \cG_{\ovr{\alpha_2}, \ovr{\beta_2}}:=\sum_{i=1}^n G_{\alpha_{2i}, \beta_{2i}}
% \]
In this article, unless otherwise stated, random variables with different indices in a summation are assumed independent. A negative binomial random variable $N_{\alpha, p}$ has the following distribution: 
\begin{equation}\label{nbrv-density}
P(N=k)=\begin{pmatrix}k+\alpha-1\\
k
\end{pmatrix}p^{\alpha}q^k, \ \ \ \    k = 0,1,2,3,4,\ldots
\end{equation}
where $\alpha > 0$ is called the \textit{shape} parameter, $p$ is called the \textit{success probability}, $0<p, q = 1-p < 1$. The geometric distribution is an important special case of the negative binomial distribution (when the shape parameter $\alpha=1$). Denote $\ovr{p_1}=(p_{11}, p_{12}, \ldots, p_{1n})$, $\ovr{p_2}=(p_{21}, p_{22}, \ldots, p_{2n})$, where $0<p_{ri}<1$, and
\[
\cN_r := \cN_{\ovr{\alpha_r}, \ovr{p_r}}:=\sum_{i=1}^n N_{\alpha_{ri}, p_{ri}}, \quad r = 1,2
\]

In this article, we study linear and arrangement constraints on the parameter $(\ovr{\alpha}, \ovr{\beta})$ that lead to $\leq_{st}$ and $\leq_{conv}$ on gamma (negative binomial) convolutions. We propose a general framework that conveniently unifies and generalizes some known results such as \cite[Theorem III.11.E.8.a]{Marshall2011}, the gamma case of \cite[Theorem 3.3]{Xu2011inequalities}, the AI (\textit{arrangement increasing}) property for tails of gamma convolutions (see the historic remark at \ref{rem:gamma-tail}), the usual stochastic orders for Pascal variables (a.k.a. negative binomial variables with integral shape parameters) in \cite{Boland1994}\cite{Zhao2010hazard} and \cite{Zhao2010pascal}. We prove the following results in Section \ref{sed:conv-order} and Section \ref{sec:st-order}, where the notations are defined in Section \ref{sec:definitions} and except $\wrcm$ is in Definition \ref{def:rcm}:

\begin{theorem}\label{intro-main-result}
\begin{enumerate}
\item \[
\vectpair{\alpha_1}{\beta_1} \geq^a \vectpair{\alpha_2}{\beta_2} \implies \cG_1 \leq_{conv} \cG_2
\]

\item 
\begin{align*}
& \vectpair{\alpha_2}{p_2} =^a (\ovr{\alpha_2}_{\uparrow}, \ovr{p_2}_{\downarrow}),\\
& \ovr{\alpha_1} \prec_w \ovr{\alpha_2}, \quad \ovr{\beta_1} \prec^w \ovr{\beta_2} \implies \cG_1\leq_{conv} \cG_2
\end{align*}

\item
\[
\vectpair{\alpha_1}{\beta_1} \wrcm \vectpair{\alpha_2}{\beta_2} \implies \cG_1\leq_{conv}\cG_2
\]

\item 
\begin{align*}
& \vectpair{\alpha_2}{p_2} =^a (\ovr{\alpha_2}_{\uparrow}, \ovr{p_2}_{\downarrow}),\\
& \ovr{\alpha_1} \prec_w \ovr{\alpha_2}, \quad \log\ovr{\beta_1} \prec^w \log\ovr{\beta_2} \implies \cG_1\leq_{st} \cG_2
\end{align*}

\item
\[
(\ovr{\alpha_1}, \log\ovr{\beta_1}) \wrcm (\ovr{\alpha_2}, \log\ovr{\beta_2}) \implies \cG_1\leq_{st}\cG_2\]
\end{enumerate}
\end{theorem}

In fact, (1) and (2) are useful special cases of (3); (4) is also a special case of (5). The partial order $\wrcm$ is defined on pairs $\vectpair{\alpha}{\beta}\in (\R^n, \R^n)$ in order to capture certain linear and arrangement operations on the parameter space that respects the $\leq_{conv}$ and $\leq_{st}$ order. Theorem \ref{intro-main-result} still holds if we replace $\ovr{\beta_{r}}$ with $\ovr{p_r}$, and $\cG_r$ with $\cN_r$, $r=1,2$. In fact, the negative binomial version of Theorem \ref{intro-main-result} is proved first, and then ``transfered'' to the gamma case with the help of a shape-mixture distribution of a MLR (\textit{monotone likelihood ratio}) family, see Section \ref{sec:section-compound-rv}.

\section{Definitions}\label{sec:definitions}
 % $x_{(1)}\leq x_{(2)}\leq \ldots \leq x_{(n)}$ be the components of $\ovr{x}$ in increasing order. Denote $\ovr{x}_{\uparrow} = (x_{(1)}, x_{(2)}, \ldots , x_{(n)})$. Let $x_{[1]}\geq x_{[2]}\geq \ldots \geq x_{[n]}$ be the components of $\ovr{x}$ in decreasing order. Denote $\ovr{x}_{\downarrow} = (x_{[1]}, x_{[2]}, \ldots , x_{[n)]})$. 
Given $\overrightarrow{x}=(\xton{x}{n}{,})\in \R^n$, let $\ovr{x}_{\uparrow} = (x_{(1)}, x_{(2)}, \ldots , x_{(n)})$ and\\ $\ovr{x}_{\downarrow} = (x_{[1]}, x_{[2]}, \ldots , x_{[n)]})$ be, respectively, the vectors with components of $\ovr{x}$ arranged in increasing (decreasing) order. For $\overrightarrow{x}, \overrightarrow{y}\in \R^n$, if
\[
\sum_{i=1}^k x_{(i)}\geq \sum_{i=1}^k y_{(i)}, \ \ \ \ k=\ton{n} 
\]
we say $\overrightarrow{x}$ is weakly majorized by $\overrightarrow{y}$ from above, denoted by $\overrightarrow{x}\prec^{w} \overrightarrow{y}$. If
\[
\sum_{i=1}^k x_{[i]}\leq \sum_{i=1}^k y_{[i]}, \ \ \ \ k=\ton{n} 
\]
we say $\overrightarrow{x}$ is weakly majorized by $\overrightarrow{y}$ from below, denoted by $\overrightarrow{x}\prec_{w} \overrightarrow{y}$. If in addition $\sum_{i=1}^n x_i=\sum_{i=1}^n y_i$, we say $\overrightarrow{x}$ is majorized by $\overrightarrow{y}$, denoted by $\overrightarrow{x}\prec \overrightarrow{y}$. It is easy to see
\[
\overrightarrow{x}\prec \overrightarrow{y} \iff \overrightarrow{x}\prec^w \overrightarrow{y}, \quad \overrightarrow{x}\prec_w \overrightarrow{y}
\] 
\[\overrightarrow{x}\prec \overrightarrow{y},\quad \overrightarrow{y}\prec \overrightarrow{x} \iff \ovr{x} \mbox{ is a permutation of } \ovr{y}\]

Given a vector $\ovr{x_1}$, we will use a double subscript $x_{1i}$ to denote the $i$-th coordinate of $\ovr{x_1}$, $i=1,\ldots, n$.

\begin{lemma}\label{T-transformations}
\cite[Section 2.19, Lemma 2]{Hardy1988inequalities}
%\cite[Lemma 2.B.1]{Marshall2011}
Given $\ovr{x}\prec \ovr{y}\in \R^n$, up to permutation, assume $\ovr{x} = \ovr{x}_{\uparrow}$, $\ovr{y} = \ovr{y}_{\uparrow}$. There exist finitely many vectors $\ovr{\phi_1}, \ovr{\phi_2}, \ldots, \ovr{\phi_k}$, $k\leq n$, such that:
\begin{enumerate}
	\item $\ovr{\phi_i} = \ovr{\phi_i}_{\uparrow}$, $i=1,\ldots, k$;
    \item $\ovr{x} = \ovr{\phi_1}$, $\ovr{y} = \ovr{\phi_k}$, and $\ovr{\phi_1}\prec \ldots \prec \ovr{\phi_k}$; 
    \item For each $s=1,\ldots, k-1$, there exist coordinates $i < j$ and a constant $\epsilon \geq 0$, where
\begin{align*}
& \phi_{s+1,i} = \phi_{s, i} - \epsilon,\quad \phi_{s+1,j} = \phi_{s, j} + \epsilon\\
& \phi_{s+1, m} = \phi_{s, m}, \quad m \neq i, j
\end{align*} 
\end{enumerate}
\end{lemma}

Consider vector pairs $(\ovr{\alpha}, \ovr{p})\in (\R^n, \R^n)$. Let $\pi$ be a permutation on $n$ letters. Obviously, 
\[
\cN = \sum_{i=1}^n N_{\alpha_i, p_i} = \sum_{i=1}^n N_{\alpha_{\pi(i)}, p_{\pi(i)}}
\]
Thus vector pairs modulo the diagonal action of permutations is a natural space to discuss the parameter configurations for convolutions of double-parameter random variables. Call this space $\cA$.  Equivalence in $\cA$ is denoted by $=^a$, in other words, $\vectpair{x_1}{y_1} = ^a \vectpair{x_2}{y_2}$ if there exists a permutation $\pi$ such that $\ovr{x_1} = \pi(\ovr{x_2})$ and $\ovr{y_1} = \pi(\ovr{y_2})$. Following \cite[Section 6.F]{Marshall2011}, we define a partial ordering called the \textit{arrangement order}, $\leq^a$ on $\cA$: $\vectpair{x_1}{y_1} \leq^a \vectpair{x_2}{y_2}$ if there exist finitely many vectors $\ovr{\phi_1}, \ldots, \ovr{\phi_k}$ such that:
\begin{enumerate}
	\item \[
	\vectpair{x_1}{y_1} =^a (\ovr{x_1}_{\uparrow}, \ovr{\phi}_1), \quad (\ovr{x_1}_{\uparrow}, \ovr{\phi}_k) =^a \vectpair{x_2}{y_2} \]

	\item For $s=1,\ldots, k-1$, $\ovr{\phi}_{s+1}$ can be obtained from $\ovr{\phi}_s$ by an interchange of two components of $\ovr{\phi}_s$, the first component is larger than the second.
\end{enumerate}
From the definition, it is easy to see that 
\begin{equation}\label{ai-relation}
(\ovr{x_1}_{\uparrow}, \ovr{y_1}_{\downarrow}) \leq^a \vectpair{x_1}{y_1} \leq^a (\ovr{x_1}_{\uparrow}, \ovr{y_1}_{\uparrow})
\end{equation}
Given a function $f: \cA\to\R$, $f$ is said to be AI (\textit{arrangement increasing}) if $f$ respects the $\leq^a$ order on the arguments. 

% Given a $2\times 2$ matrix \[
% A = \begin{pmatrix}
% a_{11} & a_{12}\\
% a_{21} & a_{22}\\
% \] 

% $A$ is said to be \textit{totally positive of order two} ($TP_2$) if $a_{ij}\geq 0$, $i, j =1,2$ and $\det(A)\geq 0$. 

Let $X$ be a random variable that is absolutely continuous with respect to the Lebesgue measure in $\R$ or the counting measure in $\Z$. Denote its probability density (mass) function by $f_X(t)$. $X_1$ is said to be less than $X_2$ in the \textit{likelihood ratio order}, denoted by $X_1\leq_{lr}X_2$, if $\frac{f_{X_2}(t)}{f_{X_1}(t)}$ is an increasing function of $t$ over the union of the support of $X_1$ and $X_2$. It is well-known that $\leq_{lr}$ implies $\leq_{st}$, see \cite{Shaked2007stochastic}. 

Gamma and negative binomial variables are \textit{infinitely divisible}, in the sense that 
\begin{equation}\label{eqn:nb-inf-div}
N_{\alpha_1, p} + N_{\alpha_2, p} = N_{\alpha_1 + \alpha_2, p}
\end{equation}
where $N_{\alpha_1, p}$ and $N_{\alpha_2, p}$ are independent. 

\begin{equation}\label{eqn:gamma-inf-div}
G_{\alpha_1, \beta} + G_{\alpha_2, \beta} = G_{\alpha_1 + \alpha_2, \beta}
\end{equation}

% The following relation is well-known \cite{Shaked2007stochastic}:
% \begin{equation}\label{eqn:relation-implication}
%   \begin{tikzcd}
%     \leq_{conv} \arrow{dr}{} &     \\
%     \leq_{lr} \arrow{r}{}  & \leq_{st} \\                                              
%   \end{tikzcd}
% \end{equation} 

\begin{lemma}\label{lem:alpha-beta-lr-order}
\begin{equation}\label{eqn:alpha-lr-order}
0 < \alpha_1 < \alpha_2 \implies N_{\alpha_1, p} \leq_{lr} N_{\alpha_2, p}, \quad G_{\alpha_1, \beta} \leq_{lr} G_{\alpha_2, \beta}
\end{equation}
\[
1 > p_1 > p_2 > 0 \implies N_{\alpha, p_1} \leq_{lr} N_{\alpha, p_2} 
\]
\[
\beta_1 > \beta_2 > 0 \implies G_{\alpha, \beta_1} \leq_{lr} G_{\alpha, \beta_2} 
\]
\end{lemma}
We omit the proof here since the verification is direct. \eqref{eqn:alpha-lr-order} says that negative binomial (gamma) random variables form a MLR (\textit{monotone likelihood ratio}) family in the shape parameter or scale parameter. This property is very useful in dealing with mixture distributions:
\begin{lemma}
\label{lem:st-mixture}
\cite[1.A.6]{Shaked2007stochastic}
Suppose $\Theta(\theta)$ is a family of distributions parametrized by $\theta$, where $\theta_1 \leq \theta_2 \implies \Theta(\theta_1) \leq_{st} \Theta(\theta_2)$. Let $Y_r = \Theta(X_r)$ be mixtures of $\Theta(\theta)$ distributions where $\theta$ is weighted by a random variable $X_r$, $r=1,2$. Then
\[X_1\leq_{st} X_2\implies Y_1\leq_{st} Y_2\]
\end{lemma}

\section{Shape Mixtures of Gamma and Negative Binomial Distributions}\label{sec:section-compound-rv}

Define $\wN_{\alpha, p}$ to be the negative binomial variable $N_{\alpha, p}$ ``shifted'' by $\alpha$:
\[
\wN_{\alpha, p} = \alpha + N_{\alpha, p}
\]

\begin{lemma}\label{shifted-nb-inf-div}
\[
\wN_{\alpha_1 + \alpha_2, p} = \wN_{\alpha_1, p} + \wN_{\alpha_2, p}
\]
\end{lemma}
\begin{proof}
By \eqref{eqn:nb-inf-div},
\begin{align*}
& \wN_{\alpha_1 + \alpha_2, p} = \alpha_1 + \alpha_2 + N_{\alpha_1 + \alpha_2, p}\\
= & \alpha_1 + N_{\alpha_1, p} + \alpha_2 + N_{\alpha_2, p} = \wN_{\alpha_1, p} + \wN_{\alpha_2, p}
\end{align*}
\end{proof}

Therefore $\wN_{\alpha, p}$ is also infinitely divisible. In the following, we will use $\wN$ to derive some technical results on shape mixtures of gamma and negative binomial random variables. Some statements in this section are proved in the appendix. To start, we will see ``the shifted negative binomial variable $\wN_{\alpha, p}$ stays the same if $p$ increases by a fixed amount and $\alpha$ increases as a certain random variable'':

\begin{proposition}\label{compound-negative-binomial}
Consider a compound random variable $\wN_{\wL, p_1}$, 
\[
P(\wN = \alpha + h + k |\wL = \alpha + h)=\begin{pmatrix}k + h + \alpha - 1\\
k
\end{pmatrix}p_1^{\alpha + h}q_1^k, \ \ \ \    k, h = 0,1,2,3,4,\ldots
\]

${\wL}$ is itself distributed as another shifted negative binomial random variable with shape parameter $\alpha$ and success probability $p_2$. Then $\wN_{{\wL}, p_1}$ has the same distribution as $\wN_{\alpha, p_1p_2}$, with shape parameter $\alpha$ and success probability $p_1p_2$, i.e.
\[\wN_{{\wL}, p_1} \stackrel{st}{=} \wN_{\alpha, p_1p_2}\]

\end{proposition}

\begin{remark}
When $L$ is a negative binomial distribution with shape parameter $\alpha$, 
\[\wN_{\wL, p} = \wL + N_{\wL, p} = \alpha + L + N_{\alpha + L, p}\]

Therefore it is possible to suppress notations such as $\wN_{\wL, p}$ and only use the usual negative binomial variables. We choose to work with the ``shifted'' version since they are more succinct. 
\end{remark}

By \eqref{eqn:nb-inf-div}, negative binomial variables are additive in the shape parameter when the scale parameters are the same, therefore Proposition \ref{compound-negative-binomial} can be extended to convolutions of negative binomial variables:
\begin{proposition}\label{prop:nb-mix-more}
Let $\widetilde{\cN} = \sum_{i=1}^n \wN_{\alpha_i, p_i}$, then \[
\widetilde{\cN} \stackrel{st}{=} \wN_{\widetilde{\cL}, p}
\]
where $p > \max_{i=1,\ldots, n} p_i$, $p'_i =\frac{p_i}{p}$, $\widetilde{\cL} = \sum_{i=1}^n \wL_{\alpha_i, p'_i}$. 
\end{proposition}

\begin{proof}
By Proposition \ref{compound-negative-binomial} and Equation \eqref{eqn:nb-inf-div}, 
\begin{align*}
& \widetilde{\cN} = \sum_{i=1}^n \wN_{\alpha_i, p_i} =  \sum_{i=1}^n \wN_{\wL_{\alpha_i, p'_i}, p} = \wN_{\sum_{i=1}^n \wL_{\alpha_i, p'_i}, p} =  \wN_{\widetilde{\cL}, p}
\end{align*}
\end{proof}

Next, consider the sum of two negative binomial mixtures with a common weight variable $\wL$ for the shape parameter, 
\[
\wN_{{\wL}, c_0+\lambda_1}+\wN_{\wL, c_0-\lambda_1}
\]
% where 
% \[\wN_{{\wL}, c_0+\lambda_1}+\wN_{\wL, c_0-\lambda_1} | \wL = \alpha + h \stackrel{st}{=} \wN_{\alpha+h{{\wL}}, c_0+\lambda_1}+\wN_{\alpha+h, c_0-\lambda_1}, \quad h=0,1,\ldots\]
conditioning on $\wL=\alpha + h$, $\wN_{\alpha + h, c_0+\lambda_1}$ and $\wN_{\alpha + h, c_0 - \lambda_1}$ are independent. Assume ${c_0 \pm \lambda_r}$, $r=1,2$ are constants between $0$ and $1$, 

%e.g. $\cN_{\cL, c_0+\lambda_1}$ is a negative binomial random variable with parameter $r=\cL$ and success probability $c_0+\lambda_1$. 

\begin{proposition}\label{compound-negative-binomial-pair}
\[
\wN_{{\wL}, c_0+\lambda_1} + \wN_{\wL, c_0-\lambda_1} \stackrel{st}{=} \wN_{{\alpha}, c_0+\lambda_2}+\wN_{\alpha, c_0-\lambda_2}, \quad 0<\lambda_2 < \lambda_1 < c_0
\]
where $\wL$ has shape parameter $\alpha$, success probability $p=\frac{c_0^2-\lambda_2^2}{c_0^2-\lambda_1^2}$.

\end{proposition}

%We will prove Theorem \ref{conv-conditons-gammma} using Theorem \ref{conditions} and the following lemma:
For shape mixtures of gamma variables, we have
\begin{proposition}\label{compound-single}
Consider a compound random variable $G_{\wL,\beta}$, where $\wL$ is a shifted negative binomial variable with shape parameter $\alpha$ and successful probability $p$. Then 
\[
G_{\wL, \beta} \stackrel{st}{=} G_{\alpha, p\beta}
\]

%In other words, 
%\[ f_{\alpha, \beta}(x)=\sum_{k=0}^{\infty} f_{(\alpha + k, \frac{\beta}{p})}(x) P(L=k) \]
\end{proposition}

The following is an analogue of Proposition \ref{prop:nb-mix-more}:
\begin{proposition}\label{prop:gamma-mix-more}
Let $\widetilde{\cG} = \sum_{i=1}^n G_{\alpha_i, \beta_i}$, then 
\[\widetilde{\cG} \stackrel{st}{=} G_{\widetilde{\cL}, \beta}\]
where $\beta > \max_{i=1,\ldots, n} \beta_i$, $p_i =\frac{\beta_i}{\beta}$, $\widetilde{\cL} = \sum_{i=1}^n \wL_{\alpha_i, p_i}$.
\end{proposition}

\section{The Convolution Order}\label{sed:conv-order}

We recall the following useful fact:

\begin{lemma}\label{conv-additive} 
\cite[Section1.1.D]{Shaked2007stochastic} Suppose $(X_1, X_2)$, $(Y_1, Y_2)$ are pairs of independent random variables, 
%When working with the convolution order for random variables, 
\begin{equation*}
X_1\leq_{conv} Y_1,\quad X_2 \leq_{conv} Y_2 \implies X_1 + X_2\leq_{conv} Y_1+Y_2
\end{equation*}
\end{lemma}

We can show the following results on the convolution order:
\begin{proposition}\label{raise-alpha}
\[0 < \alpha_1 < \alpha_2,\quad p > 0 \implies N_{\alpha_1, p} \leq_{conv} N_{\alpha_2, p}\]
\end{proposition}
\begin{proof}
This follows directly from \eqref{eqn:nb-inf-div}:
\[N_{\alpha_2, p} = N_{\alpha_1, p} + N_{\alpha_2 - \alpha_1, p} \geq_{conv} N_{\alpha_1, p}\]
\end{proof}

\begin{proposition}\label{lower-beta}
\[p_1 > p_2,\quad \alpha > 0 \implies N_{\alpha, p_1} \leq_{conv} N_{\alpha, p_2}\]
\end{proposition}
\begin{proof}
Shifting by the same constant $\alpha$, it suffices to prove
\[\wN_{\alpha, p_1} \leq_{conv} \wN_{\alpha, p_2}\]
Define $p:=\frac{p_{2}}{p_{1}}$, $\wL=\wN_{\alpha, p}$. By Proposition \ref{compound-negative-binomial}, 
\begin{align*}
& \wN_{\alpha, p_{2}} = \wN_{\wL, p_{1}} = \wN_{\alpha, p_{1}} + \wN_{L, p_{1}}
\end{align*}
where $\wN_{L, p_{1}}$ is $0$ when $L$ takes the value $0$. Obviously, $\wN_{L, p_{1}}$ is always non-negative. Therefore $N_{\alpha, p_{1}}\leq_{conv} N_{\alpha, p_{2}}$. 
\end{proof}

\begin{proposition}\label{majorize-beta}
\[\alpha > 0,\quad (p_{11}, p_{12}) \prec (p_{21}, p_{22}) \implies N_{\alpha, p_{11}} + N_{\alpha, p_{12}} \leq_{conv} N_{\alpha, p_{21}} + N_{\alpha, p_{22}}\]
\end{proposition}
\begin{proof}
Shifting by the same constant $2\alpha$, it suffices to prove
\begin{equation}\label{shifted-beta-pair}
\wN_{\alpha, p_{11}} + \wN_{\alpha, p_{12}} \leq_{conv} \wN_{\alpha, p_{21}} + \wN_{\alpha, p_{22}}
\end{equation}
Given $(p_{11}, p_{12}) \prec (p_{21}, p_{22})$, we can find positive constants $0\leq \lambda_1 \leq \lambda_2 < c_0$ such that $p_{11}, p_{12} = c_0\pm \lambda_1$; $p_{21}, p_{22} = c_0\pm \lambda_2$. When $\lambda_1 = \lambda_2$, $N_{\alpha, p_{11}} + N_{\alpha, p_{12}} = N_{\alpha, p_{21}} + N_{\alpha, p_{22}}$ and there is nothing to prove. In the following, we assume $\lambda_1 < \lambda_2$. Following Proposition \ref{compound-negative-binomial-pair}, define $p = \frac{p_{21}p_{22}}{p_{11}p_{12}} = \frac{c_0^2 -\lambda_2^2}{c_0^2 -\lambda_1^2} < 1$. $\wL = \wN_{\alpha, p}$. 
\begin{align*}
& \wN_{\alpha, p_{21}} + \wN_{\alpha, p_{22}} = \wN_{\wL, p_{11}} + \wN_{\wL, p_{12}}\\
= & \wN_{\alpha, p_{11}} + \wN_{\alpha, p_{12}} + \wN_{L, p_{11}} + \wN_{L, p_{12}}
\end{align*}
Obviously, $\wN_{L, p_{11}} + \wN_{L, p_{12}}$ is non-negative, therefore \eqref{shifted-beta-pair} is true.
\end{proof}

\begin{proposition}\label{diff-alpha-majorize-beta}
\begin{align*}
& 0\leq \alpha_1 \leq \alpha_2,\\
& p_{11} \geq p_{12}, p_{21} \geq p_{22}, \quad (p_{11}, p_{12}) \prec (p_{21}, p_{22})\implies \numberthis \label{p-conditions-maj-beta}\\ 
& N_{\alpha_1, p_{11}} + N_{\alpha_2, p_{12}} \leq_{conv} N_{\alpha_1, p_{21}} + N_{\alpha_2, p_{22}} \numberthis \label{eqn:diff-alpha_result}
\end{align*}
\end{proposition}
\begin{proof}
Under conditions in \eqref{p-conditions-maj-beta}, $p_{12} \geq p_{22}$, therefore by Proposition \ref{lower-beta}, 
\begin{equation}\label{delta-p}
N_{\alpha_2 - \alpha_1, p_{12}} \leq_{conv} N_{\alpha_2 - \alpha_1, p_{22}}
\end{equation}

Since 
\[N_{\alpha_1, p_{11}} + N_{\alpha_2, p_{12}} = N_{\alpha_1, p_{11}} + N_{\alpha_1, p_{12}} + N_{\alpha_2 - \alpha_1, p_{12}}\]
\[N_{\alpha_1, p_{21}} + N_{\alpha_2, p_{22}} = N_{\alpha_1, p_{21}} + N_{\alpha_1, p_{22}} + N_{\alpha_2 - \alpha_1, p_{22}}\]
\eqref{eqn:diff-alpha_result} follows from Proposition \ref{majorize-beta}, Equation \eqref{delta-p} and Lemma \ref{conv-additive}.
\end{proof}

\begin{proposition}\label{majorize-alpha}
\begin{align*}
& \alpha_{11} \leq \alpha_{12}, \alpha_{21} \leq \alpha_{22}, \quad (\alpha_{11}, \alpha_{12}) \prec (\alpha_{21}, \alpha_{22}), \numberthis \label{majorize-alpha-cond}\\
& p_1 \geq p_2 \implies \\
& N_{\alpha_{11}, p_{1}} + N_{\alpha_{12}, p_{2}} \leq_{conv} N_{\alpha_{21}, p_{1}} + N_{\alpha_{22}, p_{2}} \numberthis \label{majorize-alphaa-result}
\end{align*}
\end{proposition}
\begin{proof}
Under conditions \eqref{majorize-alpha-cond}, write $\alpha_{11} = \alpha_{21} + \epsilon$, $\alpha_{22} = \alpha_{12} + \epsilon$. 
\[N_{\alpha_{11}, p_{1}} + N_{\alpha_{12}, p_{2}} = N_{\alpha_{21}, p_{1}} + N_{\alpha_{12}, p_{2}} + N_{\epsilon, p_1}
\]
\[N_{\alpha_{21}, p_{1}} + N_{\alpha_{22}, p_{2}} = N_{\alpha_{21}, p_{1}} + N_{\alpha_{12}, p_{2}} + N_{\epsilon, p_2}\]
Since $p_1 \geq p_2$, by Proposition \ref{lower-beta}, $N_{\epsilon, p_1}\leq_{conv} N_{\epsilon, p_2}$, therefore \eqref{majorize-alphaa-result} follows from Lemma \ref{conv-additive}.
\end{proof}

\begin{theorem}\label{thm:conv-AI}
If $(\ovr{\alpha_1}, \ovr{p_1}) \leq^a (\ovr{\alpha_2}, \ovr{p_2})$, then $\cN_1\geq_{conv}\cN_2$.
\end{theorem}

\begin{proof}
By the definition of $\leq^a$, it suffices to prove Theorem \ref{thm:conv-AI} when $\alpha_{11} = \alpha_{22} < \alpha_{12} = \alpha_{21}$, $p_{11} = p_{22} > p_{12} = p_{21}$; $\alpha_{1i} = \alpha_{2i}$, $p_{1i} = p_{2i}$, $i=3,\ldots, n$. By Lemma \ref{conv-additive}, it further reduces to prove
\begin{equation}\label{conv-ar2}
N_{\alpha_{11}, p_{11}} + N_{\alpha_{12}, p_{12}} \geq_{conv} N_{\alpha_{12}, p_{11}} + N_{\alpha_{11}, p_{12}}
\end{equation}

Denote $p:=\frac{p_{12}}{p_{11}}$. By Proposition \ref{compound-negative-binomial}, 
\[N_{\alpha_{12}, p_{12}} = N_{\wL_{\alpha_{12}, p}, p_{11}}\]
Therefore
\[ N_{\alpha_{11}, p_{11}} + N_{\alpha_{12}, p_{12}} = N_{(\alpha_{11} + \alpha_{12} + L_{\alpha_{12}, p}), p_{11}}\]
Similarly, \[
N_{\alpha_{11}, p_{12}} + N_{\alpha_{12}, p_{11}} = N_{\alpha_{11} + \alpha_{12} + L_{\alpha_{11}, p}, p_{11}}
\]

By Proposition \ref{raise-alpha}, $N_{\alpha_{12}, p}\geq_{conv} N_{\alpha_{11}, p}$. In other words, $N_{\alpha_{12}, p}= N_{\alpha_{11}, p} + Z$, where $Z$ is nonnegative and independent of $N_{\alpha_{11}, p}$. Thus\begin{align*}
N_{\alpha_{11}, p_{11}} + N_{\alpha_{12}, p_{12}} = N_{(\alpha_{11} + \alpha_{12} + L_{\alpha_{12}, p}) +Z, p_{11}} = N_{\alpha_{11}, p_{12}} + N_{\alpha_{12}, p_{11}} + N_{Z, p_{11}}
\end{align*}   
which implies \eqref{conv-ar2}. 
\end{proof}

\begin{definition}\label{def:rcm}
Inspired by Proposition \ref{diff-alpha-majorize-beta}, \ref{majorize-alpha} and Theorem \ref{thm:conv-AI}, we consider a partial order on $\cA$ called the \textit{reverse-coupled majorization order} $\rcm$ that is generated by the following relations between $\vectpair{x_1}{y_1}$ and $\vectpair{x_2}{y_2}$:
\begin{enumerate}
% \item[(i)] \begin{equation}
% \vectpair{x_1}{y_1} \geq^a \vectpair{x_2}{y_2} 
% \end{equation}

\item[(i)] $\exists 1\leq i < j\leq n$, such that \begin{equation}\label{eqn:reverse_pair}
(y_{2j} - y_{2i})(x_{2j} - x_{2i}) \leq 0
\end{equation}
And
\begin{align*}
& (x_{1i}, x_{1j}) \prec (x_{2i}, x_{2j}), \quad x_{1m} = x_{2m}, m\neq i, j;\\
& \ovr{y_1} = \ovr{y_2} \numberthis \label{eqn:rcm-1}
\end{align*}

\item[(ii)] Under condition \eqref{eqn:reverse_pair}, 
\begin{align*}
& (y_{1i}, y_{1j}) \prec (y_{2i}, y_{2j}), \quad y_{1m} =y_{2m}, m\neq i, j;\\
& \ovr{x_1} = \ovr{x_2}  \numberthis \label{eqn:rcm-2}
\end{align*}
\end{enumerate}

Define $\vectpair{x_1}{y_1}\rcm\vectpair{x_2}{y_2}$ if we can find finitely many elements $\vectpair{\phi_i}{\psi_i}\in \cA$, $i=1,2,\ldots, k$ such that $\vectpair{x_1}{y_1} = \vectpair{\phi_1}{\psi_1}$, $\vectpair{x_2}{y_2} = \vectpair{\phi_k}{\psi_k}$; and $\vectpair{\phi_i}{\psi_i}$, $(\ovr{\phi}_{i+1}, \ovr{\psi}_{i+1})$, $i=1,\ldots, k-1$ satisfy one of the above relations (i-ii).

In addition, consider
\begin{enumerate}
\item[(iii)]
\[x_{1i} \leq x_{2i},\quad i=1,\ldots, n; \quad \ovr{y_1} = \ovr{y_2}\]
%\[ \]

\item[(iv)]
\[y_{1i} \geq y_{2i},\quad i=1,\ldots, n; \quad \ovr{x_1} = \ovr{x_2} \]
%\[\]
\end{enumerate}	

We call a partial ordering generated by $\rcm$ together with relations (iii-iv) the \textit{reverse-coupled weakly sub-sup majorization order}, denoted by $\vectpair{x_1}{y_1}\wrcm \vectpair{x_2}{y_2}$. 
\end{definition}
\begin{remark}
From the definition, it is straightforward to see
\begin{equation*}
\begin{split}
\vectpair{x_1}{y_1}\rcm\vectpair{x_2}{y_2} & \implies \ovr{x_1}\prec \ovr{x_2}, \quad \ovr{y_1}\prec \ovr{y_2}\\
\vectpair{x_1}{y_1}\wrcm\vectpair{x_2}{y_2} & \implies \ovr{x_1}\prec_w \ovr{x_2}, \quad \ovr{y_1}\prec^w \ovr{y_2}\\
\end{split}
\end{equation*}
Because of the restriction in \eqref{eqn:reverse_pair}, the reverse implications are only true in special circumstances, for example if $\ovr{x_1} = \ovr{x_2} = (\alpha, \ldots, \alpha)$, then:
\begin{align}
& \ovr{y_1}\prec \ovr{y_2}\implies \vectpair{x_1}{y_1}\rcm\vectpair{x_2}{y_2} \label{rcm-reverse1}\\
& \ovr{y_1}\prec^w \ovr{y_2}\implies \vectpair{x_1}{y_1}\wrcm\vectpair{x_2}{y_2} \label{rcm-reverse2}
\end{align}

Similarly, if $\ovr{y_1} = \ovr{y_2} = (\beta, \ldots, \beta)$, then:
\begin{align}
& \ovr{x_1}\prec \ovr{x_2}\implies \vectpair{x_1}{y_1}\rcm\vectpair{x_2}{y_2} \label{rcm-reverse3}\\
& \ovr{x_1}\prec_w \ovr{x_2}\implies \vectpair{x_1}{y_1}\wrcm\vectpair{x_2}{y_2} \label{rcm-reverse4}
\end{align}

In addition, in Relation (i) let $(x_{1i}, x_{1j})$ be a permutation of $(x_{2i}, x_{2j})$, then it follows from the definition of $\leq^a$ that 
\begin{equation}\label{ai-implies-rcm}
\vectpair{x_1}{y_1} \geq^a \vectpair{x_2}{y_2} \implies \vectpair{x_1}{y_1}\rcm\vectpair{x_2}{y_2} 
\end{equation}
From \eqref{rcm-reverse1} \eqref{rcm-reverse2} \eqref{rcm-reverse3} \eqref{rcm-reverse4} and \ref{ai-implies-rcm}, we conclude that $\rcm$ and $\wrcm$ are extensions of the majorization (weak majorization) relations from vectors in $\R^n$ to vector pairs in $\cA$. Meanwhile, $\rcm$ and $\wrcm$  can also be considered as extensions of $\leq^a$ (with the inequalities reversed) between vector pairs with same components (up to rearrangements) to vector pairs with different components. In this sense, $\rcm$ and $\wrcm$ incorporate both majorization (weak majorization) and arrangement informations in a ``coupled'' way.
\end{remark}

We can capture the results obtained so far in the following general statement:
\begin{theorem}\label{thm:conv-general}
\begin{equation}
\vectpair{\alpha_1}{p_1} \wrcm \vectpair{\alpha_2}{p_2} \implies \cN_1\leq_{conv}\cN_2
\end{equation}
\end{theorem}

\begin{proof}
We will first show $\cN_1\leq_{conv} \cN_2$ when $\vectpair{\alpha_1}{p_1}$, $\vectpair{\alpha_2}{p_2}$ satisfy Relation (i) in Definition \ref{def:rcm}. Fix $1\leq i \leq j \leq n$, without loss of generality, assume $p_{1i} = p_{2i} \leq p_{2j} = p_{1j}$, $\alpha_{2i} > \alpha_{2j}$. Given $(\alpha_{1i}, \alpha_{1j})\prec (\alpha_{2i}, \alpha_{2j})$, if $\alpha_{1i} \geq \alpha_{1j}$, then from Proposition \ref{majorize-alpha}, \[
N_{\alpha_{1i}, p_{1i}} + N_{\alpha_{1j}, p_{1j}} \leq_{conv} N_{\alpha_{2i}, p_{2i}} + N_{\alpha_{2j}, p_{2j}}
\]
Since $\alpha_{1m} = \alpha_{2m}$, $p_{1m} = p_{2m}$, $m\neq i, j$, $\cN_1\leq_{conv}\cN_2$ follows from Lemma \ref{conv-additive}. 

If $\alpha_{1i} < \alpha_{1j}$, define $\ovr{\alpha_{1}}'$ by 
\[
\alpha'_{1i} = \alpha_{1j},\quad \alpha'_{1j} = \alpha_{1i}; \quad \alpha_{1m} = \alpha'_{1m}, m\neq i, j
\]

Then $\vectpair{\alpha_1}{p_1} \geq^a (\ovr{\alpha_1}', \ovr{p_1})$, therefore by Theorem \ref{thm:conv-AI} and the argument above,
\[\cN_1 \leq_{conv} \cN_{\ovr{\alpha_1}', \ovr{p_1}} \leq_{conv} \cN_2\]

The proof for Relation (ii) of Definition \ref{def:rcm} is similar, where Proposition \ref{diff-alpha-majorize-beta} is used instead of Proposition \ref{majorize-alpha}. The proofs for Relation (iii) and (iv) of Definition \ref{def:rcm} follow easily from Proposition \ref{raise-alpha} and \ref{lower-beta}.
\end{proof} 

From \eqref{ai-implies-rcm}, Theorem \ref{thm:conv-AI} can be seen as a special case of Theorem \ref{thm:conv-general}.

\begin{theorem}\label{conv-conditons-gammma}
\[\vectpair{\alpha_1}{\beta_1} \wrcm \vectpair{\alpha_2}{\beta_2} \implies \cG_1\leq_{conv}\cG_2\]
\end{theorem}

\begin{proof}
%(Proof of Theorem \ref{conv-conditons-gammma}) 
Following Proposition \ref{prop:nb-mix-more}, pick a large constant $\beta > \max_{r=1,2; i=1,\ldots, n}\beta_{ri}$, define $p_{ri} = \frac{\beta_{ri}}{\beta}$, we have $\cG_{r} = G_{R_r + \cN_r, \beta}$ where
\[ 
R_r = \sum_{i=1}^n \alpha_{ri},\quad \cN_r = \sum_{i=1}^n N_{\alpha_{ri}, p_{ri}}
\] 
\[\vectpair{\alpha_1}{\beta_1} \wrcm \vectpair{\alpha_2}{\beta_2}\implies \vectpair{\alpha_1}{p_1} \wrcm \vectpair{\alpha_2}{p_2} \implies \cN_1\leq_{conv} \cN_2\]
Therefore $\cN_2 = \cN_1 + Z$ where $Z$ is a non-negative r.v. independent of $\cN_1$. Moreover, $\ovr{\alpha_1}\prec_w \ovr{\alpha_2} \implies R_1 \leq R_2$, therefore  
\[ 
\cG_2 = G_{R_2 + \cN_2, \beta} = G_{R_1 + \cN_1, \beta} + G_{R_2 - R_1 + Z, \beta} = \cG_1 + G_{R_2 - R_1 + Z, \beta} 
\]
Since $G_{R_2 - R_1 + Z, \beta}\geq 0$,
\[
\cG_1\leq_{conv} \cG_2
\] 

\end{proof}

\begin{example}
For example, $\vectpair{\alpha_1}{\beta_1} \rcm \vectpair{\alpha_2}{\beta_2}$ where  
%\begin{align*}
\[\ovr{\alpha_1} = (0.4, 0.6, 0.5), \quad  \ovr{p_1} = (2,3,4) \]
\[\ovr{\alpha_2} = (0.7, 0.3, 0.5), \quad  \ovr{p_2} = (1,3,5) \]
%\end{align*}
This can be seen with the help of intermediate pairs,
\[\ovr{\phi_2} = (0.4, 0.6, 0.5), \quad \ovr{\psi_2} = (2,2,5) \]
\[\ovr{\phi_3} = (0.7, 0.3, 0.5), \quad \ovr{\psi_3} = (2,2,5) \]
That $\vectpair{\alpha_1}{\beta_1}\rcm \vectpair{\phi_2}{\psi_2}$ follows from (ii) of Definition \ref{def:rcm}, where $i=2$, $j=3$; $\vectpair{\phi_2}{\psi_2}\rcm \vectpair{\phi_3}{\psi_3}$ follows from (i), where $i=1$, $j=2$; finally, $\vectpair{\phi_3}{\psi_3} \rcm \vectpair{\alpha_2}{\beta_2}$ follows from (ii), where $i=1$, $j=2$. Therefore $\cG_{\ovr{\alpha_1}, \ovr{\beta_1}}\leq_{conv} \cG_{\ovr{\alpha_2}, \ovr{\beta_2}}$. 
\end{example}

In the following, we give an easily recongnizable criterion for $\wrcm$ in a useful special case:
\begin{theorem}
\label{thm:weak-prec-cond}
If $\vectpair{x_2}{y_2} =^a (\ovr{x_2}_{\uparrow}, \ovr{y_2}_{\downarrow})$,
then
\begin{equation}\label{weak-prec-nb}
\ovr{x_1} \prec \ovr{x_2}, \quad \ovr{y_1} \prec \ovr{y_2} \implies \vectpair{x_1}{y_1} \rcm \vectpair{x_2}{y_2}
\end{equation}
\begin{equation}
\ovr{x_1} \prec_w \ovr{x_2}, \quad \ovr{y_1} \prec^w \ovr{y_2} \implies \vectpair{x_1}{y_1} \wrcm \vectpair{x_2}{y_2}
\end{equation}
\end{theorem}

Theorem \ref{thm:weak-prec-cond} is proved in the appendix. 

%$\ovr{\alpha_1} = \ovr{\alpha_1}_{\uparrow}$, $\ovr{p_1} = \ovr{p_1}_{\downarrow}$, 
\begin{corollary}\label{cor:weak-prec}
Suppose $\vectpair{\alpha_2}{p_2} =^a (\ovr{\alpha_2}_{\uparrow}, \ovr{p_2}_{\downarrow})$,
\begin{equation*}
\ovr{\alpha_1} \prec_w \ovr{\alpha_2}, \quad \ovr{p_1} \prec^w \ovr{p_2} \implies \cN_1\leq_{conv} \cN_2
\end{equation*}
Similarly, if $\vectpair{\alpha_2}{\beta_2} =^a (\ovr{\alpha_2}_{\uparrow}, \ovr{\beta_2}_{\downarrow})$,
\begin{equation*}
\ovr{\alpha_1} \prec_w \ovr{\alpha_2}, \quad \ovr{\beta_1} \prec^w \ovr{\beta_2} \implies \cG_1\leq_{conv} \cG_2
\end{equation*}
\end{corollary}

We note that when $\ovr{\alpha_2}=(\alpha, \ldots, \alpha)$ or $\ovr{p_2} = (p,\ldots, p)$, then $\vectpair{\alpha_2}{p_2} =^a (\ovr{\alpha_2}_{\uparrow}, \ovr{p_2}_{\downarrow})$ is trivially satisfied. 

\begin{remark}
In Corollary \ref{cor:weak-prec}, we only need $\vectpair{\alpha_2}{\beta_2} =^a (\ovr{\alpha_2}_{\uparrow}, \ovr{\beta_2}_{\downarrow})$ but have no such requirement for $\vectpair{\alpha_1}{\beta_1}$. This is essentially due to the result of Theorem \ref{thm:conv-AI}. It is interesting to compare the condition in Corollary \ref{cor:weak-prec} with known results in the literature. For example, using our notation, when $\alpha_{ri}\geq 1$, $r=1,2$, $i=1,\ldots, n$, if $\vectpair{\alpha_1}{\beta_1} =^a (\ovr{\alpha_1}_{\uparrow}, \ovr{\beta_1}_{\downarrow})$, $\vectpair{\alpha_2}{\beta_2} =^a (\ovr{\alpha_2}_{\uparrow}, \ovr{\beta_2}_{\downarrow})$, $\ovr{\alpha_1}\prec \ovr{\alpha_2}$ and $\ovr{\beta_1}\prec^w \ovr{\beta_2}$, \cite{Zhao2010pascal} and \cite{Zhao2011} obtained the likelihood ratio ordering for convolutions of gamma and Pascal variables.   
\end{remark}

\begin{remark}\label{rem:alternate-wrcm}
Without loss of generality, assume $y_{1i} > y_{1j}$ and $(y_{1i}, y_{1j})\prec^w (y_{2i}, y_{2j})$. We can find $z\geq y_{1i}$ such that $(z, y_{1j})\prec (y_{2i}, y_{2j})$. Using this idea, it is easy to see that the following relations generate a partial order that is equivalent to $\wrcm$:
\begin{enumerate}
% \item[(i)] \begin{equation}
% \vectpair{x_1}{y_1} \geq^a \vectpair{x_2}{y_2} 
% \end{equation}

\item[(i)] Under condition \eqref{eqn:reverse_pair},
\begin{align*}
& (x_{1i}, x_{1j}) \prec_w (x_{2i}, x_{2j}), \quad x_{1m} = x_{2m}, m\neq i, j;\\
& \ovr{y_1} = \ovr{y_2} %\numberthis \label{eqn:rcm-1-alt}
\end{align*}

\item[(ii)] Under condition \eqref{eqn:reverse_pair}, 
\begin{align*}
& (y_{1i}, y_{1j}) \prec^w (y_{2i}, y_{2j}), \quad y_{1m} =y_{2m}, m\neq i, j;\\
& \ovr{x_1} = \ovr{x_2}  %\numberthis \label{eqn:rcm-2-alt}
\end{align*}

\item[(iii)]
\[x_{1i} \leq x_{2i},\quad i=1,\ldots, n; \quad \ovr{y_1} = \ovr{y_2}\]
%\[ \]

\item[(iv)]
\[y_{1i} \geq y_{2i},\quad i=1,\ldots, n; \quad \ovr{x_1} = \ovr{x_2} \]
\end{enumerate}
It should be noted that (iii) and (iv) does not follow from (i) and (ii) because of the restriction \eqref{eqn:reverse_pair}.
\end{remark}

\section{the Usual Stochastic Order}\label{sec:st-order}

We prove an analogue of Proposition \ref{majorize-beta},

\begin{proposition}
\begin{align*}
& \alpha > 0, \quad (\log p_{11}, \log p_{12})\prec (\log p_{21}, \log p_{22}) \\  
& \implies N_{\alpha, p_{11}} + N_{\alpha, p_{12}} \leq_{st} N_{\alpha, p_{21}} + N_{\alpha, p_{22}} 
\numberthis\label{eqn:log-prec2} 
\end{align*}
\end{proposition}

\begin{proof}
Shifting by the same constant $2\alpha$, it suffices to prove
\[\wN_{\alpha, p_{11}} + \wN_{\alpha, p_{12}} \leq_{st} \wN_{\alpha, p_{21}} + \wN_{\alpha, p_{22}}\]

Without loss of generality, assume $p_{21} \leq p_{11}\leq p_{12}\leq p_{21}$. Denote $p_1:=\frac{p_{11}}{p_{22}}$, $p_2:=\frac{p_{12}}{p_{22}}$, $p_{3}:=\frac{p_{21}}{p_{22}}$. Then $p_3=p_1p_2$ if $(\log p_{11}, \log p_{12})\prec (\log p_{21}, \log p_{22})$. By Proposition \ref{compound-negative-binomial}, \[
\wN_{\alpha, p_{11}} + \wN_{\alpha, p_{12}} = \wN_{\wL_{\alpha, p_1}, p_{22}} + \wN_{\wL_{\alpha, p_2}, p_{22}} = \wN_{\wL_{\alpha, p_1} + \wL_{\alpha, p_2}, p_{22}}
\] 

Similarly, \[
\wN_{\alpha, p_{21}} + \wN_{\alpha, p_{22}} = \wN_{\wL_{\alpha, p_3} + \alpha, p_{22}} 
\]

Therefore by Lemma \ref{lem:st-mixture}, it suffices to prove 
\begin{equation}\label{eqn:st-prod}
\wL_{\alpha, p_1} + \wL_{\alpha, p_2} \leq_{st} \wL_{\alpha, p_3} + \alpha 
\end{equation}

By Proposition \ref{compound-negative-binomial} again, 
\begin{align*}
& \alpha + \wL_{\alpha, p_3} = \alpha + \wL_{\wL_{\alpha, p_1}, p_2} \\
& =  \alpha + \wL_{\alpha + L_{\alpha, p_1}, p_2}\\
& =  \alpha + \wL_{L_{\alpha, p_1}, p_2} + \wL_{\alpha, p_2}\\
& =  \alpha + L_{\alpha, p_1} + L_{L_{\alpha, p_1}, p_2} + \wL_{\alpha, p_2}\\
& =  \wL_{\alpha, p_1} + \wL_{\alpha, p_2} + L_{L_{\alpha, p_1}, p_2}\\
& \geq_{st} \wL_{\alpha, p_1} + \wL_{\alpha, p_2} \numberthis\label{eqn:st-last-line}
\end{align*}
where the last step of \eqref{eqn:st-last-line} is true because $\wL_{\alpha, p_2}\geq 0$.
\end{proof}

\begin{remark}
In the proof of \eqref{eqn:st-last-line}, we notice a key difference between the usual stochastic order and the convolution order: since $L_{L_{\alpha, p_1}, p_2}$ is not independent from $L_{\alpha, p_1}$, we can not strengthen the usual stochastic ordering in \ref{eqn:log-prec2} to the convolution order. 
\end{remark}
The following fact is evident from the definition,
\begin{lemma}\label{st-additive}
\begin{equation}
X_1\leq_{st} Y_1,\quad X_2 \leq_{st} Y_2 \implies X_1 + X_2\leq_{st} Y_1+Y_2
\end{equation} 
\end{lemma}

Using Lemma \ref{st-additive} and similar ideas from the proof of Proposition \ref{diff-alpha-majorize-beta}, we can show the following:

\begin{proposition}\label{diff-alpha-majorize-beta-st}
\begin{align*}
& 0\leq \alpha_1 \leq \alpha_2,\\
& p_{11} \geq p_{12}, p_{21} \geq p_{22}, \quad (\log p_{11}, \log p_{12}) \prec (\log p_{21}, \log p_{22})\implies \\%\numberthis \label{p-conditions-maj-beta-st}\\ 
& N_{\alpha_1, p_{11}} + N_{\alpha_2, p_{12}} \leq_{conv} N_{\alpha_1, p_{21}} + N_{\alpha_2, p_{22}} %\numberthis \label{eqn:diff-alpha_result-st}
\end{align*}
\end{proposition}

Proposition \ref{raise-alpha}, \ref{lower-beta} and \ref{majorize-alpha} readily apply to the usual stochastic order, which is implied by the convolution order. Moreover, we can replace $p_{ri}$ with $\log p_{ri}$, $r,i = 1,2$ in the conditions of these propositions, since $\log$ is an increasing function. Following ideas in the proofs of Theorem \ref{thm:conv-general} and Theorem \ref{conv-conditons-gammma}, we readily obtain the following result:
% Combined with Proposition \ref{diff-alpha-majorize-beta-st}
\begin{theorem}\label{st-order-both}
\[(\ovr{\alpha_1}, \log \ovr{p_1}) \wrcm (\ovr{\alpha_2}, \log \ovr{p_2}) \implies \cN_1\leq_{st}\cN_2\]
where $\log \ovr{p}=(\log p_1, \ldots, \log p_n)$.
\[(\ovr{\alpha_1}, \log \ovr{\beta_1}) \wrcm (\ovr{\alpha_2}, \log \ovr{\beta_2}) \implies \cG_1\leq_{st}\cG_2\]
\end{theorem}

\begin{remark}
Since $\log$ is an increasing concave function, by \cite[Theorem I.5.A.2]{Marshall2011}, 
\[
(p_{11}, p_{12})\prec (p_{21}, p_{22}) \implies (\log p_{11}, \log p_{12})\prec^w (\log p_{21}, \log p_{22}) 
\]
Therefore from Remark \ref{rem:alternate-wrcm}, it is easy to see  
\[(\ovr{\alpha_1}, \ovr{p_1}) \wrcm (\ovr{\alpha_2}, \ovr{p_2}) \implies (\ovr{\alpha_1}, \log \ovr{p_1}) \wrcm (\ovr{\alpha_2}, \log \ovr{p_2})\]
Since the convolution order is a strong ordering that implies the usual stochastic order, it is not surprising that the conditions in Theorem \ref{thm:conv-general}, \ref{conv-conditons-gammma} are stronger than that in Theorem \ref{st-order-both}. 
\end{remark}

Similar to Corollary \ref{cor:weak-prec}, we have the following:
\begin{corollary}
\label{cor:weak-prec-st}
Suppose $\vectpair{\alpha_2}{p_2} =^a (\ovr{\alpha_2}_{\uparrow}, \ovr{p_2}_{\downarrow})$,
\[\ovr{\alpha_1} \prec_w \ovr{\alpha_2}, \quad \log\ovr{p_1} \prec^w \log\ovr{p_2} \implies \cN_1\leq_{st} \cN_2\]

Similarly, if $\vectpair{\alpha_2}{\beta_2} =^a (\ovr{\alpha_2}_{\uparrow}, \ovr{\beta_2}_{\downarrow})$,
\[\ovr{\alpha_1} \prec_w \ovr{\alpha_2}, \quad \log\ovr{\beta_1} \prec^w \log\ovr{\beta_2} \implies \cG_1\leq_{st} \cG_2\]
\end{corollary}

\begin{remark}
We briefly review some results in the literature that are related to Theorem \ref{st-order-both}. Using a simple change of variables, the following result is obtained in \cite[Theorem III.11.E.8.a]{Marshall2011}: 
\begin{equation}\label{proschan-result}
\ovr{\alpha_1} = \ovr{\alpha_2} = (\alpha, \ldots, \alpha), \quad \alpha > 0; 
\quad \log \overrightarrow{\beta_1} {\prec} \log\overrightarrow{\beta_2} \implies \cG_1 \leq_{st} \cG_2 
\end{equation} 
\eqref{proschan-result} is generalized by \cite{Yu2011weighted}\cite{Xu2011inequalities}\cite{Pan2013II}, which also develop a technique to work with a more general class of continuous distributions that includes the gamma distribution. Let $\lambda_i =\frac{1}{\beta_i}$, $i=1,2,\ldots, n$. It is straightforward to check that
\[\cG = \sum_{i=1}^n \lambda_i G_{\alpha_i, 1} = \sum_{i=1}^n G_{\alpha_i, \beta_i}\]

% \begin{equation}
% G_{\alpha_1, 1}\leq_{lr}\ldots \leq_{lr} G_{\alpha_n, 1}
% \end{equation} 

When $\alpha_1 < \alpha_2 < \ldots < \alpha_n$, by \eqref{eqn:alpha-lr-order} $G_{\alpha_1, 1}\leq_{lr}\ldots \leq_{lr} G_{\alpha_n, 1}$. Following \cite[Theorem 3.3, Remark 3.3]{Xu2011inequalities}, \cite[Corollary 4.7]{Pan2013II} 
\begin{align*}
& \vectpair{\alpha_1}{\beta_1} =^a (\ovr{\alpha_1}_{\uparrow}, \ovr{\beta_1}_{\downarrow}),\quad \vectpair{\alpha_2}{\beta_2} =^a (\ovr{\alpha_2}_{\uparrow}, \ovr{\beta_2}_{\downarrow})\\
& \ovr{\alpha_1} = \ovr{\alpha_2}, \quad \log\ovr{\beta_1}\prec^w \log\ovr{\beta_2} \numberthis \label{pan-conditions}\\
& \implies \cG_{1}\leq_{st} \cG_2
\end{align*}

See also \cite[Theorem 3.3]{Zhao2011} for a condition that leads to the hazard rate order (a stronger order that implies the usual stochastic order) when $\alpha_{ri}\geq 1$, $r=1,2$. In the discrete case, the usual stochastic order for convolutions of Pascal variables is studied in \cite{Boland1994}\cite{Yu2009}\cite{Zhao2010hazard}\cite{Zhao2010pascal} and the references therein. Finally, the arrangement increasing property of certain stochastic orderings is discussed in \cite{Ma2000convex}. 

Corollary \ref{cor:weak-prec-st} and Theorem \ref{st-order-both} extended the above results to more general conditions for convolution of gamma and negative binomial random variables. 
\end{remark}
% Let $\lambda_i =\frac{1}{\beta_i}$, $i=1,2,\ldots, n$. It is straightforward to check that
% \[\cG = \sum_{i=1}^n \lambda_i G_{\alpha_i, 1} = \sum_{i=1}^n G_{\alpha_i, \beta_i}\]
%As an immediate corollary of Theorem \ref{thm:conv-AI}, we have he following:
\begin{corollary}\label{gamma-tail-is-AI}
For any $c>0$, $P(\sum_{i=1}^n \lambda_i G_{\alpha_i, 1} \geq c)$ is an AI function of $\vectpair{\alpha}{\lambda}$.
\end{corollary}

\begin{proof}
Since $\lambda_i =\frac{1}{\beta_i}$ is a decreasing function, it is easy to see \[
\vectpair{\alpha_1}{\lambda_1} \leq^a \vectpair{\alpha_2}{\lambda_2} \iff \vectpair{\alpha_1}{\beta_1} \geq^a \vectpair{\alpha_2}{\beta_2}\]

From Theorem \ref{thm:conv-AI}, 
\begin{align*}
\sum_{i=1}^n \lambda_{1i} G_{\alpha_{1i}, 1} \leq_{conv} \sum_{i=1}^n \lambda_{2i} G_{\alpha_{2i}, 1}\implies \\
\sum_{i=1}^n \lambda_{1i} G_{\alpha_{1i}, 1} \leq_{st} \sum_{i=1}^n \lambda_{2i} G_{\alpha_{2i}, 1}\implies \\
P(\sum_{i=1}^n \lambda_{1i} G_{\alpha_{1i}, 1} \geq c) \leq P\big(\sum_{i=1}^n \lambda_{2i} G_{\alpha_{2i}, 1} \geq c\big)
\end{align*}  
\end{proof}

\begin{remark}\label{rem:gamma-tail}
Corollary \ref{gamma-tail-is-AI} can also be proved using other known results in the literature. By \eqref{eqn:alpha-lr-order} and \cite[Lemma 2.2(c)]{Hollander1977}, the density $\prod_{i=1}^n f_{G_{\alpha_i, 1}}(t_i)$ is AI in $(\ovr{\alpha}, \ovr{t})$. On the other hand, for fixed constant $c$ the indicating function $I_{\sum_{i=1}^n \lambda_i t_i\geq c}$ is AI in $(\ovr{t}, \ovr{\lambda})$. Since the AI property is preserved under convolution, 
\[P(\sum_{i=1}^n \lambda_i G_{\alpha_i, 1} \geq c) = \int_{\sum_{i=1}^n \lambda_i t_i\geq c}\prod_{i=1}^nf_{G_{\alpha_i, 1}}(t_i)dt_1\ldots dt_n\]
is AI in $(\ovr{\alpha}, \ovr{\lambda})$, see \cite[Example 3.2.2]{boland1988AI}. 
\end{remark}

We will further apply results in this article to study other stochastic orderings such as the likelihood ratio order in a subsequent article, which will among other things show that in general, the condition $\vectpair{\alpha_1}{\beta_1} \wrcm \vectpair{\alpha_2}{\beta_2}$ cannot be weakened to $(\ovr{\alpha_1}, \log\ovr{\beta_1}) \wrcm (\ovr{\alpha_2}, \log\ovr{\beta_2})$ for the convolution order.  

\appendix

\section{Proofs}

The following lemma is used in the proofs for statements in Section \ref{sec:section-compound-rv}. We omit its proof here since the verification is direct:
\begin{lemma}
The probability generating function of $\wN_{\alpha, p}$ is 
\begin{equation}\label{mgf-nb}
E(t^{\wN{\alpha, p}})=\sum_{k=0}^{\infty} t^{k+\alpha} P(N=k) =\bigg(\frac{p}{\frac{1}{t}-q}\bigg)^{\alpha}
\end{equation}

The moment generating function is
\begin{equation}\label{shifted-ng-mgf}
E(e^{t\wN{\alpha, p}})=\bigg(\frac{p}{e^{-t}-q}\bigg)^{\alpha}
\end{equation}
\end{lemma}

% \begin{proof}[Proof of Lemma \ref{shifted-nb-inf-div}]
% By \ref{shifted-ng-mgf},
% \[E(t^{\wN{\alpha_1 + \alpha_2, p}}) = E(t^{\wN{\alpha_1, p}})E(t^{\wN{\alpha_2, p}}) = E(t^{\wN{\alpha_1, p} + \wN{\alpha_2, p}})
% \]
% \end{proof}

In many cases, we can prove an equality of the law of random variables by proving their moment generating functions are the same. This is justified when the probability law corresponding to a given moment sequence is unique. A sufficient condition is the following, which is satisfied by all probability laws occurring in this article:

\begin{theorem}\cite[Theorem 30.1]{Billingsley1995}\label{analytic-mgf}
Let $\mu$ be a probability measure on the line having finite moments of all orders 
\[
m_k=\int_{-\infty}^{\infty} x^k \mu(dx), \ \ \ \ k=1,2,\ldots
\] 

If the power series $\sum_{k=1}^{\infty}\frac{m_kt^k}{k!}$ has a positive radius of convergence, then $\mu$ is uniquely determined by its moments $m_k$, $k=1,2,\ldots$. Namely, $\mu$ is the only probability measure with moments $m_k$, $k=1,2,\ldots$
\end{theorem}

\begin{proof}[Proof of Proposition \ref{compound-negative-binomial}]
Consider the moment generating function of $\wN_{\wL, p_1}$:
\begin{align*}
E(e^{t\wN_{\wL, p_1}}) = & \sum_{k_1 = 0}^{\infty} E(e^{\wN_{\alpha + k_1}, p_1})P(\wL = \alpha + k_1)\\
= & \sum_{k_1=0}^{\infty} \left(\frac{p_1}{e^{-t} - q_1}\right)^{\alpha + k_1}P(\wL = \alpha + k_1)\\
= & E(u^{\wL})\bigg\rvert_{u=\frac{p_1}{e^{-t}-q_1}}\\
= & \left(\frac{p_2}{\frac{1}{u}-q_2}\right)^{\alpha}_{u=\frac{p_1}{e^{-t}-q_1}}\\
= & \left(\frac{p_1p_2}{e^{-t}-(1-p_1p_2)}\right)^{\alpha}\\
= & E(e^{t\wN_{\alpha, p_1p_2}})
\end{align*}

%which matches the moment generating function of a shifted negative binomial variate with parameter $\alpha$, success probability $p_1p_2$.

\end{proof}

\begin{proof}[Proof of Proposition \ref{compound-negative-binomial-pair}]
Again we will consider the m.g.f.:
\begin{align*}
& E\big[e^{t(\wN_{\wL, c_0+\lambda_1} + \wN_{\wL, c_0-\lambda_1})}\big] \\
= & \sum_{k=0}^{\infty} E\big[e^{t(\wN_{k+\alpha, c_0+\lambda_1}+\wN_{k+\alpha, c_0-\lambda_1})}\big]P(\wL=k+\alpha)\\
 = & \sum_{k=0}^{\infty} \left(\frac{c_0+\lambda_1}{e^{-t}-(1-c_0-\lambda_1)}\right)^{\alpha+k} \left(\frac{c_0-\lambda_1}{e^{-t}-(1-c_0+\lambda_1)}\right)^{\alpha+k} P(\wL=\alpha + k)\\
= & \sum_{k=0}^{\infty} \left[\frac{c_0^2-\lambda_1^2}{(e^{-t}-1+c_0)^2-\lambda_1^2}\right]^{\alpha+k} P(\wL=\alpha+k) \\ \label{plugin-nb2-2}
 = & \left(\frac{p}{\frac{1}{u}-q}\right)^{\alpha}\bigg\rvert_{u=\frac{c_0^2-\lambda_1^2}{(e^{-t}-1+c_0)^2-\lambda_1^2}}\\ 
 = & \left[\frac{p(c_0^2-\lambda_1^2)}{(e^{-t}-1+c_0)^2-\lambda_1^2-q(c_0^2-\lambda_1^2)}\right]^{\alpha}\\
 = & \left[\frac{c_0^2-\lambda_2^2}{(e^{-t}-1+c_0)^2-\lambda_2^2}\right]^{\alpha}\\
 = & E[e^{t(\wN_{\alpha, c_0+\lambda_2}+\wN_{\alpha, c_0-\lambda_2})}]
\end{align*}
%where \ref{plugin-nb2} is by equation \ref{mgf-nb}.
\end{proof}

\begin{proof}[Proof of Proposition \ref{compound-single}]

Given a Gamma variate $G_{\alpha, \beta}$ with shape parameter $\alpha$ and scale parameter $\beta$, it is straightforward to compute that 
\[ 
E(e^{tG})=\bigg(\frac{1}{1-\frac{t}{\beta}}\bigg)^{\alpha} 
\]

We consider the moment generating function of $G_{\wL, \beta}$:
\begin{align*}
E(e^{tG_{\wL, \beta}}) & = \sum_{k=0}^{\infty} E(e^{tG_{\alpha+k, \beta}}) P(\wL = k+\alpha)\\ 
& = \sum_{k=0}^{\infty} \bigg(\frac{1}{1-\frac{t}{\beta}}\bigg)^{\alpha+k}  P(\wL = k+\alpha)\\ 
& = \bigg(\frac{p}{\frac{1}{u}-q}\bigg)^{\alpha}\bigg\rvert_{u=\frac{1}{1-\frac{t}{\beta}}}\\
& = \bigg(\frac{p}{p-\frac{t}{\beta}}\bigg)^{\alpha}\\
& = \bigg(\frac{1}{1-\frac{t}{p\beta}}\bigg)^{\alpha}\\
& = E(e^{tG_{\alpha, p\beta}})
\end{align*}
%\label{plug-in-nb}
%where \ref{plug-in-nb} is by equation \ref{mgf-nb}. Obviously, $(\frac{1}{1-\frac{t}{\frac{\beta}{p}}})^{\alpha}$ is the m.g.f. of a Gamma variate with shape parameter $\alpha$, scale parameter $p\beta$.
\end{proof}

We include here an easy corollary of Proposition \ref{compound-negative-binomial-pair} and \ref{compound-single}:
\begin{corollary}%\label{compound-negative-binomial-pair}
\[
G_{{\wL}, c_0+\lambda_1} + G_{\wL, c_0-\lambda_1} \stackrel{st}{=} G_{{\alpha}, c_0+\lambda_2} + G_{\alpha, c_0-\lambda_2}, \quad 0<\lambda_2 < \lambda_1 < c_0
\]
where $\wL$ has shape parameter $\alpha$, success probability $p=\frac{c_0^2-\lambda_2^2}{c_0^2-\lambda_1^2}$.
\end{corollary}

\begin{proof}
Pick $\beta > c_0 + \lambda_1$, by Proposition \ref{compound-single}, 
\begin{align*}
& G_{{\alpha}, c_0+\lambda_2} + G_{\alpha, c_0-\lambda_2} =  G_{\big(\wN_{\alpha, \frac{1}{\beta}(c_0 -\lambda_2)}\big), \beta } + G_{\big(\wN_{\alpha, \frac{1}{\beta}(c_0 + \lambda_2)}\big), \beta }\\
= & G_{\big(\wN_{\alpha, \frac{1}{\beta}(c_0 -\lambda_2)} + \wN_{\alpha, \frac{1}{\beta}(c_0 + \lambda_2)}\big), \beta}\\
= & G_{\big(\wN_{\wL, \frac{1}{\beta}(c_0 -\lambda_1)} + \wN_{\wL, \frac{1}{\beta}(c_0 + \lambda_1)}\big), \beta} \\
= & G_{{\wL}, c_0+\lambda_1} + G_{\wL, c_0-\lambda_1}
\end{align*}

\end{proof}

\begin{proof}[Proof of Theorem \ref{thm:weak-prec-cond}]
%We will only prove \eqref{weak-prec-nb}. 
Without loss of generality, assume $\ovr{x_1} = \ovr{x_1}_{\uparrow}$, $\ovr{x_2} = \ovr{x_2}_{\uparrow}$, $\ovr{y_2} = \ovr{y_2}_{\downarrow}$. Define $\ovr{x_1}'$ where 
\[
x'_{1n} = x_{1n} + \sum_{i=1}^n (x_{2i} - x_{1i}) \geq x_{1n}; \quad x'_{1i} = x_{1i}, i = 1, 2,\ldots, n-1\]
Then $\ovr{x_{1}}'\prec \ovr{x_2}$ and $\ovr{x_1}' = \ovr{x_1}'_{\uparrow}$. Similarly, define $\ovr{y_2}'$ where 
\[
y'_{21} = y_{21} + \sum_{i=1}^n (y_{2i} - y_{1i}) \geq y_{21}; \quad y'_{2i} = y_{2i}, i\neq j\]
Then $\ovr{y_{1}}\prec \ovr{y_2}'$ and $\ovr{y_2}' = \ovr{y_2}'_{\downarrow}$. By Relation (iii) and (iv) of Definition \ref{def:rcm}, 
\[
(\ovr{x_{1}}, \ovr{y_1}) \wrcm (\ovr{x_{1}}', \ovr{y_1}),\quad ({\ovr{x_{2}}, \ovr{y_2}'}) \wrcm ({\ovr{x_{2}}, \ovr{y_2}})\] 

Therefore it suffices to prove \eqref{weak-prec-nb}. We will show
\begin{equation}
\label{eqn:tmp1}
(\ovr{x_{1}}', \ovr{y_1}) \rcm (\ovr{x_{1}}', \ovr{y_1}_{\downarrow}) \rcm (\ovr{x_{1}}', \ovr{y_2}') \rcm (\ovr{x_{2}}, \ovr{y_2}')
\end{equation}
Since $\ovr{x_1}'=\ovr{x_1}'_{\uparrow}$, $\ovr{x_2}=\ovr{x_2}_{\uparrow}$, $\ovr{x_1}'\prec \ovr{x_2}$, we consider finitely many vectors $\ovr{\phi_i}$, $i=1,\ldots, k$ as in Lemma \ref{T-transformations}. For $s=1,\ldots, k$, since $\ovr{\phi_s} = \ovr{\phi_s}_{\uparrow}$, $\ovr{y_2}'=\ovr{y_2}'_{\downarrow}$, for any $1\leq i, j\leq n$,
\[(\phi_{si} - \phi_{sj})(y'_{2i} - y'_{2j})\leq 0\]
By Relation (i) of Definition \ref{def:rcm} 
\[(\ovr{\phi_s}, \ovr{y_2}')\rcm (\ovr{\phi}_{s+1}, \ovr{y_2}'), \quad s=1,2,\ldots, k-1\]

By transitivity, $(\ovr{x_{1}}', \ovr{y_2}') \rcm (\ovr{x_{2}}, \ovr{y_2}')$. The proof for $(\ovr{x_{1}}', \ovr{y_1}_{\downarrow}) \rcm (\ovr{x_{1}}', \ovr{y_2}')$ is similar. Finally, $(\ovr{x_{1}}', \ovr{y_1}) \rcm (\ovr{x_{1}}', \ovr{y_1}_{\downarrow})$ by \eqref{ai-relation} and \eqref{ai-implies-rcm}.
\end{proof}

\section*{Acknowledgment}

The author is grateful for Louis Gordon from Susquehanna International Group for introducing him to the topic of stochastic orders, and for many helpful discussions and comments on this article. 

\bibliographystyle{abbrv}
\bibliography{ref3_snapshot}

\end{document}

%% file: preambles.tex
\usepackage{amssymb,amscd, amsmath}
\usepackage{mathtools,centernot,tensor}
\input{CourseNotesPreamble}

\allowdisplaybreaks

\newcommand{\vectpair}[2]{(\ovr{#1}, \ovr{#2})}
\newcommand{\rcm}{\stackrel{rc}{\prec}}
\newcommand{\wrcm}{\stackrel{\ rc}{\tensor[_w]{\prec}{^w}}}

\def\ton#1{1,\ldots, {#1}}

\def\xton#1#2#3{{#1}_1 {#3} {#1}_2{#3} \ldots{#3} {#1}_{#2}}

%% file: CourseNotesPreamble.tex
 \usepackage{amsmath}       % I think this gives me some symbols
 \usepackage{amsthm}        % Does theorem stuff
 \usepackage{amssymb}       % more symbols and fonts

\usepackage{enumerate} %% To customize enumeration.
\usepackage{hyperref}
\newcounter{homework}
\newcounter{project}

 \renewcommand{\theequation}{\thesection.\arabic{equation}}         %%
  \renewcommand{\thehomework}{\thesection.\arabic{homework}}
   
 \makeatletter                                                      %%
    \@addtoreset{equation}{section} % Make the equation counter reset each section
    \@addtoreset{footnote}{section} % Make the footnote counter reset each section
    \@addtoreset{homework}{section} % Make the homework counter reset each section
    \@addtoreset{project}{section} % Make the project counter reset each section
    {\hspace*{\fill}$\lrcorner$\endgraf\endgroup\end{trivlist}}     %%

 \def\newprooflikeenvironment#1#2#3#4{%                             %%
      \newenvironment{#1}[1][]{%                                    %%
          \refstepcounter{equation}                                 %%
          \begin{proof}[{\rm\csname#4\endcsname{#2~\theequation}%   %%
          \@ifnotempty{##1}{\the\thm@notefont \ (##1)}\csname#4\endcsname{.}}]%%
          \def\qedsymbol{#3}}%                                      %%
         {\end{proof}}}                                             %%
 \makeatother                                                       %%
 \theoremstyle{plain}               %%%%% Theorem-like commands
 \newtheorem{theorem}[equation]{Theorem}                            %%
 \newtheorem*{lemma*}{Lemma}                                        %%
 \newtheorem*{theorem*}{Theorem}                                    %%
 \newtheorem{lemma}[equation]{Lemma}                                %%
\newtheorem{corollary}[equation]{Corollary}                        %%
 \newtheorem{proposition}[equation]{Proposition}                    %%

 \theoremstyle{definition} %%%% Definition-like Commands
\newtheorem{definition}[equation]{Definition}
\newtheorem{example}[equation]{Example}

 \newprooflikeenvironment{remark}{Remark}{$\diamond$}{textbf}       %%
 \newprooflikeenvironment{digression}{Digression}{$\diamond$}{textbf}       %%
\theoremstyle{remark} %%%%%% Remark-Like Commands
 \DeclareFontFamily{OT1}{pzc}{}                                 %%
 \DeclareFontShape{OT1}{pzc}{m}{it}{<-> s * [1.100] pzcmi7t}{}  %%
 \DeclareMathAlphabet{\mathpzc}{OT1}{pzc}{m}{it}                %%
 \DeclareFontFamily{U}{manual}{}                                %%
 \DeclareFontShape{U}{manual}{m}{n}{ <->  manfnt }{}            %%
 \newcommand{\manfntsymbol}[1]{%                                %%
    {\fontencoding{U}\fontfamily{manual}\selectfont\symbol{#1}}}%%
 \newcommand{\bbar}[1]{\setbox0=\hbox{$#1$}\dimen0=.2\ht0 \kern\dimen0 \overline{\kern-\dimen0 #1}}

 \DeclareMathOperator{\End}{\ensuremath{\mathcal{E}\kern-.125em\mathpzc{nd}}}

 \DeclareMathOperator{\Hom}{\mathcal{H}\kern-.125em\mathpzc{om}}

 \DeclareMathOperator{\Proj}{\mathcal{P}\kern-.125em\mathpzc{roj}}

 \newcommand{\udot}{\ensuremath{{\lower .183333em \hbox{\LARGE \kern -.05em$\cdot$}}}}

 \newcommand{\cA}{\mathcal{A}}

 \newcommand{\cG}{\mathcal{G}}

 \newcommand{\cL}{\mathcal{L}}
 
 \newcommand{\cN}{\mathcal{N}}

 \newcommand{\R}{\mathbb{R}}

  \newcommand{\Z}{\mathbb{Z}}

 \write0{\today}